\documentclass[12pt]{amsart}
\include{macros}
\allowdisplaybreaks

\usepackage{paralist} 

\usepackage{geometry}

\usepackage{amsmath}

\usepackage{amssymb}
\usepackage{amsfonts}
\usepackage{color}
\usepackage{enumerate}

\usepackage{booktabs}

\usepackage{etoolbox}
\patchcmd{\thebibliography}{\section*{\refname}}{}{}{}

\usepackage{verbatim}

\usepackage{graphicx}
\usepackage{subfig}
\usepackage{bm}
\usepackage{tikz}
\bgroup

\definecolor{darkviolet}{rgb}{0.58,0,0.83} 

\usepackage[bookmarks,bookmarksopen=false,
                  pdfauthor={Autor},%
                  pdftitle={Titel},%
                  colorlinks=true,%
                  linkcolor=blue,%
                  citecolor=blue,%
                  urlcolor=blue,%
                ]{hyperref}

\newtheorem{theorem}{Theorem}

\newtheorem{proposition}[theorem]{Proposition}
\newtheorem{corollary}[theorem]{Corollary}

\newtheorem{remark}[theorem]{Remark}

\DeclareMathOperator*{\esssup}{ess\,\sup}

\newcommand{\pbox}{\hfill$\Box$\\}

\newcommand{\C}{{\mathcal{C}o_\pi}\,}
\newcommand{\CiP}[1]{{\mathcal{C}o_{\pi_1}L^{#1}_{1/w_1}(G_1)}\,}
\newcommand{\CiiP}[1]{{\mathcal{C}o_{\pi_2}L^{#1}_{1/w_2}(G_2)}\,}
\newcommand{\Ci}[1]{{\mathcal{C}o_{\pi_1}L^{#1}_{w_1}(G_1)}\,}
\newcommand{\Cii}[1]{{\mathcal{C}o_{\pi_2}L^{#1}_{w_2}(G_2)}\,}

\newcommand{\Cim}[1]{{\mathcal{C}o_{\pi_1}L^{#1}_{m_1}(G_1)}\,}
\newcommand{\Ciim}[1]{{\mathcal{C}o_{\pi_2}L^{#1}_{m_2}(G_2)}\,}

\newcommand{\cD}{\mathcal{D}}
\newcommand{\R}{\mathbb{R}}
\newcommand{\CC}{\mathbb{C}}
\newcommand{\rd}{\mathbb{R}^d}
\newcommand{\rdd}{\mathbb{R}^{2d}}

\newcommand{\Z}{\mathbb{Z}}

\newcommand{\I}{\mathcal{I}}
\newcommand{\HS}{\mathcal{HS}}

\renewcommand{\H}{\mathcal{H}}

\newcommand{\dom}{{\sf Dom}\,}
\newcommand{\ran}{{\sf Ran}\,}

\definecolor{darkviolet}{rgb}{0.58,0,0.83} 

\newcommand{\modsp}{modulation space}

\renewcommand{\leq}{\leqslant}

\renewcommand{\geq}{\geqslant}
\newcommand{\paff}{\pi_{\mathrm{aff}}}
\newcommand{\pwh}{\pi_{\mathrm{WH}}}
\begin{document}

\begin{abstract}
 We prove  general kernel theorems for operators acting between
 coorbit spaces. These are Banach spaces associated to an integrable
 representation of a locally compact group and contain most of the
 usual function spaces (Besov spaces, modulation spaces, etc.). A 
 kernel  theorem describes the  form of every bounded operator between
 a coorbit  space of test functions and distributions by means of a
 kernel in a coorbit space associated to the tensor product
 representation. As special cases we recover Feichtinger's kernel
 theorem for modulation spaces and the recent generalizations
 by  Cordero and Nicola. We also obtain a kernel theorem for operators
 between the Besov spaces $\dot{B}^0_{1,1}$ and $\dot{B}^{0}_{\infty, \infty }$.  
\end{abstract}

\date{}
\title{Kernel Theorems in  Coorbit Theory}
\author{Peter Balazs}
\address{Acoustics Research Institute\\Austrian Academy of Sciences\\Wohllebengasse 12-14, 1040 Vienna\\ Austria }
\email{peter.balazs@oeaw.ac.at}
\author{Karlheinz Gr\"ochenig}
\address{Faculty of Mathematics 
University of Vienna 
Oskar-Morgenstern-Platz 1 
A-1090 Vienna, Austria}
\email{karlheinz.groechenig@univie.ac.at}
\author{Michael Speckbacher}
\address{Institut de Math{\'e}matiques de Bordeaux \\Universit{\'e} de Bordeaux \\351, cours de la Lib{\'e}ration - F 33405 TALENCE \\ France }
\email{speckbacher@kfs.oeaw.ac.at}

\subjclass[2010]{42B35, 42C15, 46A32, 47B34}
\date{}
\keywords{Kernel theorems, coorbit theory, continuous frames, operator representation, tensor products, Hilbert-Schmidt operators}
\thanks{P.B. and M.S. were 
  supported in part by the  START-project FLAME (’Frames and Linear Operators for Acoustical Modeling and Parameter Estimation’; Y 551-N13)   of the
Austrian Science Fund (FWF), and K.\ G.\ was
  supported in part by the  project P31887-N32  of the
Austrian Science Fund (FWF)}
\maketitle

\section{Introduction}\label{sec:intro}

Kernel theorems assert that every ``reasonable'' operator can written
 as a ``generalized'' integral operator. For instance, the Schwartz
 kernel theorem states that a continuous linear  
operator $A:\mathcal{S}(\R^d)\rightarrow\mathcal{S}^\prime(\R^d)$ possesses
a unique distributional kernel $K\in \mathcal{S}^\prime(\R^{2d})$, such that
\begin{equation}
  \label{eq:c1}
\langle Af,g\rangle= \langle K,g\otimes
f \rangle,\
f,g\in\mathcal{S}(\R^d) \, . 
\end{equation}
If $K$ is a locally
integrable function, then
$$
\langle Af, g \rangle =\int_{\R^d}K(x,y)f(y) \overline{g(x)} dy dx , \
f, g \in\mathcal{S}(\R^d), 
$$
and thus $A$  has indeed the form of an integral operator. Similar
kernel theorems hold for continuous operators from $\mathcal{D}(\rd )
\to \cD ^\prime (\rd )$ \cite[Theorem 5.2]{hoe90} and for Gelfand-Shilov spaces and
their distribution spaces \cite{geshi64}. The importance of these kernel
theorems stems from the fact that they offer a general formalism for
the description of linear operators.

In the context of time-frequency analysis,
Feichtinger's kernel theorem \cite{feichtinger80cras} (see also
\cite{FGchui} and \cite[Theorem 14.4.1]{groe1}) states that every bounded linear operator
from the \modsp\ $M^1(\rd )$ to the \modsp\ $M^\infty (\rd )$ can be represented in the form
\eqref{eq:c1} with a kernel in $M^\infty (\rdd )$. The advantage of
this kernel theorem is that both the space of test functions $M^1(\rd
)$ and the distribution space $M^\infty (\rd ) = M^1(\rd )^*$ are
Banach spaces and thus technically easier than the locally convex
spaces $\mathcal{S}(\R ^d)$ and $\mathcal{S}^\prime (\R ^d)$. 

Recently, Cordero and Nicola \cite{coni17} revisited   Feichtinger's kernel
theorem  and proved several new  kernel theorems that ``{\em do not have a
counterpart in distribution theory}''. They argue that ``{\em this reveals
the superiority, in some respects, of the modulation space formalism
upon distribution theory}''. While we agree full-heartedly with this
claim, we would like to add a more abstract point of view and argue
that the deeper reason for this superiority  lies in the theory of
coorbit spaces  and in the convenience of Schur's test for integral
operators. Indeed, we 
will  prove kernel
theorems similar to Feichtinger's kernel theorem  for many
coorbit spaces. 

The main idea  is to investigate operators in a transform domain, after
taking  a short-time Fourier transform, a wavelet transform, or an
abstract wavelet transform, i.e., a continuous transform with respect
to a unitary group 
representation. In this new representation every operator between a
suitable space of test functions and distributions is  an integral
operator. The standard  boundedness conditions of Schur's
test then yield strong kernel theorems.

The technical framework for this idea is coorbit  theory which was  introduced and studied in
\cite{fegr88,fegr89,fegr89II,groe91} for the construction and analysis
of function spaces by means of a 
generalized wavelet transform.  The main idea is that functions
 in the  standard function spaces, such as Besov spaces and
modulation spaces, can be characterized by the  decay or integrability
properties of an associated transform (the wavelet transform or the
short-time Fourier transform). 
In the abstract setting,  $G$ is a locally compact group and $\pi:G\rightarrow
\mathcal{U}(\H)$ is an irreducible, unitary, integrable representation
of $G$. Leaving technical details aside,  the coorbit space $\C
L^p_w(G)$ consists of all distributions $f$ in a suitable distribution
space, such that the representation coefficient  $g\mapsto \langle
f,\pi(g)\psi\rangle $ is in the weighted space $L^p_w(G)$. 

Next, let $G_1$ and $G_2$ be  two locally
compact  groups, and  $(\pi_1, \H _1)$ and $(\pi_2, \H _2) $ be
irreducible, unitary, integrable   representations of $G_1$ and $G_2$
respectively.

Let  $A$ be a bounded linear  operator between $\Ci{1}$ and
$\CiiP{\infty}$. Our main insight is that such an  operator  can be
described by a kernel in a coorbit space that is related to the tensor
product representation   $\pi=\pi_2\otimes\pi_1$ of $G=G_1\times G_2$
on the  tensor product space $\H_2\otimes \H_1$.  The following
non-technical  formulation offers  a flavor of  our main result in 
Theorem~\ref{thm-kernel-2}:

\emph{A linear operator $A$ is bounded from  $\Ci{1}$ to 
$\CiiP{\infty}$, if 
and only if there exists a kernel $K\in \C L^\infty_{w_1^{-1} \otimes w_2^{-1}}(G_1\times G_2)$ such
that  }
\begin{equation}
\langle A\upsilon,\varphi\rangle=\langle K,\varphi\otimes\upsilon\rangle,
\end{equation} 
for all  $\upsilon\in\Ci{1},\ \varphi\in\Cii{1}$.

This statement is not just a mere abstraction and  generalization of
the classical  kernel theorem. With the choice
of a specific group and representation 
one obtains explicit kernel theorems. For instance, using the
Schr\"odinger representation of the Heisenberg group,  one recovers
Feichtinger's original kernel theorem. The added value is our insight
that the  conditions on the
kernel  of \cite{coni17} in terms of mixed modulation spaces
\cite{bi10} amount to  coorbit spaces with respect to the tensor product
representation. Choosing  the $ax+b$-group and the continuous wavelet
representation, one obtains a  kernel theorem for  all bounded
operators operators between the Besov spaces $\dot{B}^0_{1,1} $ and  $\dot{B}^0 _{\infty ,\infty }$
with a kernel in a space of dominating mixed smoothness. This  class of
function spaces has been studied extensively ~\cite{SU09,ST87} and is
by no means artificial. 

By using suitable versions of Schur's test, it is then possible to
derive  characterizations for the boundedness of operators
between other coorbit spaces. For example, in
Theorem~\ref{thm-kernel-char1}  we will prove
the following, with  $\frac{1}{p} + \frac{1}{q} = 1$: 

\medskip 
\hspace{-0.3cm}\begin{tabular}{lll}
$(i)$ \ $A:\Ci{1}\rightarrow \Cii{p}$ bounded &$\Leftrightarrow$ &
$K\in \C \mathcal{L}^{p,\infty}_{ 1/w_1\otimes w_2}(G_1\times G_2)$,\\

$(ii)$ $A:\Ci{p}\rightarrow \Cii{\infty}$ bounded &$\Leftrightarrow$& 
$K\in \C L^{q,\infty}_{1/w_1\otimes w_2}(G_1 \times G_2)$,
\end{tabular}

\medskip

\noindent where the mixed-norm Lebesgue spaces $\mathcal{L}^{p,q}$ and
$L^{p,q}$ on $G_1 \times G_2$  are defined in \eqref{eq:mixed-1} and \eqref{eq:mixed-2} respectively.



The paper is organized as follows. In Section~\ref{sec:prel} we present the basics of tensor products 
 and coorbit space theory. The theory of coorbit spaces of kernels
 with respect to products of integrable representations is developed
 in  Section~\ref{sec:coorb-prod}. Our main results, the kernel
 theorems,  are proved in Section~\ref{sec:kernel} and applied to
 particular examples of group representations in Section~\ref{sec:ex}.  

We note that our proofs require a meaningful formulation of coorbit
theory. One can therefore prove kernel theorems also   in the context of
 coorbit space theories~\cite{CO11,DMLS17}, e.g., for  certain reducible
 representations. 

\section{Preliminaries on Tensor Products and Coorbit Spaces}\label{sec:prel}

\subsection{Tensor Products and  Hilbert-Schmidt Operators 
}\label{subsec:tensor-mixed}

The theory of tensor products is at the heart of kernel theorems for
operators. Algebraically, a simple tensor of two vectors (in two possibly
different Hilbert spaces) is a formal product of two vectors
$f_1 \otimes f_2$,  and the tensor product $\H_1 \otimes \H_2$ is
obtained by taking the completion of  all linear combinations  of
simple tensors with respect to the inner product
$$
\langle f_1 \otimes
f_2, g_1 \otimes g_2 \rangle := \langle f_1, g_1 \rangle \, \langle
g_2,f_2 \rangle \, .
$$
This tensor product is homogeneous in the following sense: 
    $\alpha\cdot(f_1\otimes f_2)=(\alpha f_1)\otimes f_2=f_1\otimes
    (\overline{\alpha}f_2)$. Note explicitly that the product $f_1
    \otimes f_2$ is anti-linear in the second factor. In some books
    this is done by introducing the dual Hilbert space $\H _{2}^\prime
    $~\cite{KR1-97}.

    If each  Hilbert space is an  $L^2$-space $\H_1
= L^2(X,\mu ),\ \H_2 = L^2(Y,\nu )$, then the simple tensor
$f\otimes g$ is just the product $(x,y) \mapsto$ $f(x)\cdot \overline{g(y)}$ and the tensor
product becomes the product space $ \H_1 \otimes \H_2 = L^2(X,\mu
) \otimes L^2(Y,\nu ) = L^2(X\times Y, \mu \times \nu )$.

The connection between functions and operators arises in the analytic
approach to tensor products. We interpret a function of \emph{two}
variables as an integral kernel for an operator. Thus a simple tensor
$f_1\otimes f_2$ of two functions becomes the rank one operator $f \mapsto
\langle f, f_2 \rangle f_1$ with integral kernel $ f_1(x)\overline{f_2(y)}$,
and a general $k\in L^2(X\times Y,\mu \times \nu )$ becomes a
Hilbert-Schmidt operator from $L^2(Y,\nu)$ to $L^2(X,\mu)$. The systematic, analytic treatment of
general tensor products of two Hilbert-spaces often  defines the
tensor product as a space of Hilbert-Schmidt operators between $\H
_2$ and $\H_1$. We note that his definition is already based on the
characterization of Hilbert-Schmidt operators and thus represents  a
non-trivial kernel theorem~\cite{con90}. 
Whereas the working mathematician
habitually identifies an operator with its distributions kernel, 
we will  make the conceptual  distinction  between tensor products and
operators for  our study of kernel theorems.

In the sequel we will denote the (distributional) kernel of an integral
operator  by $k$ and the abstract kernel in a tensor product by $K$.

\subsection{Coorbit Space Theory}\label{subsec:coorbit}
Let $G$ be a locally compact group with left Haar measure $ \int _G
\dots dg$, $\H$ be  a separable Hilbert
space, 
and $\mathcal{U}(\H)$ the group  of unitary operators acting on $\H$. A
continuous unitary group representation $\pi:G\rightarrow
\mathcal{U}(\H)$ is called \emph{square integrable}  \cite{DM76,alanga91-1},
if it is irreducible  and there exist $\psi\in\H$ such that  
\begin{equation}\label{square-int}
\int_G|\langle \psi,\pi(g)\psi\rangle|^2dg<\infty \, .
\end{equation}
A non-zero  vector $\psi$ satisfying \eqref{square-int} is called
\emph{admissible}. For every square integrable representation there
exist a densely defined operator $T$ such that $\forall f_1,f_2\in\H,\
\psi_1,\psi_2\in \dom(T)$, one has 
\begin{equation}\label{ortho-rel}
\int_G\langle f_1,\pi(g)\psi_1\rangle\langle \pi(g)\psi_2,f_2\rangle dg=\langle T\psi_2,T\psi_1\rangle\langle f_1,f_2\rangle.
\end{equation}

For fixed $\psi _1 = \psi _2 = \psi$  the representation coefficient
$f \mapsto V_\psi f(g):=\langle f,\pi(g)\psi\rangle$ is 
interpreted as   a
\emph{generalized wavelet transform}. The orthogonality relation
\eqref{ortho-rel} then implies that $V_\psi $ is   a multiple of an
isometry from $\H$ to $L^2(G)$. By using a weak interpretation of
vector-valued integrals, 
\eqref{ortho-rel} can also be recast as the  inversion formula
\begin{equation}\label{inversion}
f=\frac{1}{\|T\psi\|^2}\int_G\langle f,\pi(g)\psi\rangle\pi(g)\psi dg
\, .
\end{equation}

For the rest of this paper we assume without loss of generality  that
the chosen admissible vectors $\psi$ are normalized,
i.e. $\|T\psi\|=1$.  

The adjoint operator  $V_\psi^\ast: L^2(G) \to \H $ is formally
defined  by
$$
V_\psi^\ast F:=\int_G F(g)\pi(g)\psi dg. 
$$
Other domains  and convergence properties will be discussed later.

With this notation \eqref{inversion} says that $V_\psi^\ast
V_\psi=I_\H$ for all admissible and normalized vectors $\psi $, which
in the language of recent frame theory means that $\{\pi(g)\psi\}_{g\in
  G}$  is  a continuous Parseval frame. 
By \cite[Proposition 2.1]{bow17} one can  always assume that $G$ is $\sigma$-finite since we assume $\H$ to be separable.  

In coorbit theory one needs much  stronger hypotheses on $\pi $. 
The representation $\pi$ is called \emph{integrable} with respect to a
weight $w$ if there exists an admissible vector $\psi\in \H$ such that 
\begin{equation}\label{eq-integrable}
\int_G|\langle \psi,\pi(g)\psi\rangle| w(g) \, dg<\infty.
\end{equation}
Let $g_1,g_2,g_3\in G$. We call a weight  $w:G\rightarrow \R^+$
\emph{submultiplicative}, if  {$w(g_1g_2)\leq w(g_1)w(g_2)$, and a
  function $m:G\rightarrow \R^+$ \emph{w-moderate}, if it satisfies
  $m(g_1 g_2 g_3)\leq w(g_1)m(g_2)w(g_3)$. 
If $m$ is $w$-moderate, the weighted Lebesgue space  $L^p_m(G)$ is  then
invariant under left  translation $L_xf(y)=f(x^{-1}y)$ and under the
right translation  $R_xf(y)=f(yx)$.
Throughout this paper, we will assume that
the weight $w$ satisfies 
\begin{equation}\label{assumpt-on-w}
w(x)\geq C\max\left\{\alpha(x),\alpha(x^{-1}),\beta(x),\Delta(x^{-1})\beta(x^{-1})\right\},
\end{equation}
where $\alpha(x):=\|L_x\|_{L^p_m(G)\rightarrow L^p_m(G)}$, $\beta(x):=\|R_x\|_{L^p_m(G)\rightarrow L^p_m(G)}$, and $\Delta$ denotes the modular function of $G$.

Our standing assumption is that the representation $\pi $ of $G$
possesses an admissible vector $\psi $ such that $V_\psi \psi \in
L^1_w(G)$. We denote the corresponding set by 

$$\mathcal{A}_w(G):=\left\{\psi\in\H, \psi \neq 0: \ V_\psi\psi\in L^1_w(G)\right\}.$$
For fixed $\psi \in \mathcal{A}_w(G)$ the linear version of
$\mathcal{A}_w(G)$ 
\begin{equation}
\H^1_w:=\left\{f\in\H:\ V_\psi f\in L^1_w(G)\right\}
\end{equation}
is dense in $\H$. Let $(\H^1_w)^\sim$ denote the anti-dual of
$\H^1_w$, i.e., the space of anti-linear continuous functionals on
$\H^1_w$. As $\H^1_w$ is dense in $\H $, it follows that the inner product on $\H\times\H$  extends to $(\H^1_w)^\sim\times \H^1_w$ and so does the generalized wavelet transform.

The coorbit space  with respect to $L^p_m(G)$ is then defined by
$$
\C L^p_m(G):=\left\{f\in(\H^1_w)^\sim:\ V_\psi f\in L^p_m(G)\right\}
\, ,
$$
and is  equipped with the natural norm
$$\|f\|_{\C L^p_m(G)}:=\|V_\psi
f\|_{L^p_m(G)} \, .$$
With our assumptions on $\pi, \psi, m$, the coorbit
 space   $\C L^p_m(G)$ is a Banach space~ \cite{fegr89}. Alternatively, $\C
 L^p_m(G)$ for $p<\infty$ can be defined as the completion of $H_w^1$ with respect to
 this norm. 
Moreover,
\begin{equation}
  \label{eq:f1a}
\C L^2(G)=\H,\ \ \C L^1_w(G)=\H^1_w,\ \mbox{ and } \ \C
L^\infty_{1/w}(G)=(\H^1_w)^\sim =\H^\infty_{1/w} , 
\end{equation}
and
\begin{equation}
  \label{eq:f1}
  \H^1_w \subseteq \C L^p_m(G) \subseteq \H^\infty_{1/w}, 
\end{equation}
for $1\leq p\leq \infty $ and $w$-moderate weight $m$. 
In the context of coorbit space theory the space $\H _w^1$ serves as a
space of test functions, and $\H^\infty_{1/w} $ is the corresponding
distribution space. 

We quickly recall some of the fundamental properties of coorbit spaces, see for example \cite[Theorem 4.1, Theorem 4.2 and Proposition 4.3]{fegr89}. 
\begin{proposition}\label{props-of-coorbit}Let $\psi,\phi\in \mathcal{A}_w(G)$, $f\in \C L^p_m(G)$, $g\in \C L^q_{1/m}(G)$ and $F\in L^p_m(G)$.
Then the following properties hold:
\begin{enumerate}[(i)]\item
  $V_\psi:\C L^p_m(G)\rightarrow L^p_m(G)$ is an 
  isometry. \label{ana-iso}\vspace{0.05cm}
\item $\H _m^p$ is invariant with respect to $\pi $ and $\|\pi (g) f\|
  _{\C L^p_m(G)} \leq C w (g) \| f\|
  _{\C L^p_m(G)} $ for all $g\in G, f\in \C L^p_{m}(G) $.  \vspace{0.05cm}
  \item $V_\psi^\ast:L^p_m(G)\rightarrow
    \C L^p_m(G)$ is continuous. \label{synth-cont}\vspace{0.05cm}
    \item  $V_\psi^\ast V_\psi=I_{\C L^p_m(G)}$.\vspace{0.05cm}
  \item {\em Correspondence principle}: Let  $F\in L^p_m(G)$.  There exists
  $f\in \C L^p_m(G)$ such that $F=V_\psi f$ if and only if $F=F\ast V_\psi\psi$, where $\ast$ denotes convolution on $G$. \label{corr-princ} \vspace{0.05cm}
\item {\em Duality}: For  $1\leq p<\infty$, $\frac{1}{p}+\frac{1}{q}=1$, we have  $(\C L^p_m(G))^* = \C
  L^{q}_{1/m}(G)$, where   the duality is given by
  $$
  \langle f, g \rangle_{\C L^p_m(G), \C
  L^{q}_{m}(G)} = \langle V_\psi f , V_\psi g \rangle _{L^p_m(G),
    L^{q}_{1/m}(G)} \, .\vspace{0.05cm}
  $$\label{duality}
\vspace{-0.45cm}  
\item The definition of $\C L^p_m(G)$ is independent
of the particular choice of the window function from
$\mathcal{A}_w(G)$. In particular, $\|V_\psi
f\|_{L^p_m(G)}\asymp\|V_\phi f\|_{L^p_m(G)}$ for arbitrary non-zero $\phi ,\psi
\in \mathcal{A}_w(G)$. \vspace{0.05cm}
 \end{enumerate}

\end{proposition}

We furthermore need a result on the existence of atomic decompositions for  the  space $\C L_{w}^1(G)$, see 
 \cite[Theorem 4.7]{fegr88}.
\begin{theorem}\label{prop-equ-norms}
Let $\psi \in \mathcal{A}_w(G)$. There exists a discrete
subset $\{g_i\}_{i\in\mathcal{I}}\subset G$ and a collection of linear
functionals  $\lambda_i:\C  L^1_w(G)\rightarrow { \CC},\ i\in\mathcal{I}$ such that
\begin{equation}\label{expansion}
f=\sum_{i\in\mathcal{I}}\lambda_i(f)\pi(g_i)\psi,\mbox{ with }\sum_{i\in\mathcal{I}}|\lambda_i(f)|w(g_i)\asymp \|f\|_{\C L^1_w(G)},
\end{equation}
and the sum converges absolutely in $\C  L^1_w(G)$.
\end{theorem}

\section{Frames and Coorbit Spaces via Tensor Products}\label{sec:coorb-prod}

Let $G_1,G_2$ be two locally compact groups with unitary square
integrable representations $\pi_1:G_1\rightarrow\mathcal{U}(\H_1)$ 
and $\pi_2:G_2\rightarrow \mathcal{U}(\H_2)$. For  $g:=(g_1,g_2)\in
G:=G_1\times G_2$ the tensor representation  $\pi:G\rightarrow
\mathcal{U}(\H_2\otimes\H_1)$,  
$$
\pi(g):=\pi_2(g_2)\otimes \pi_1(g_1),
$$
acts on a simple tensor $\Psi:=\psi_2\otimes \psi_1\in\H_2\otimes\H_1$ by
\begin{equation}\label{prod-rep-tens}
\pi(g)(\psi_2\otimes\psi_1)=\pi_2(g_2)\psi_2\otimes\pi_1(g_1)\psi_1.
\end{equation}
It   follows  immediately that $\pi$ is a unitary 
 representation of $G$ on $\H_2\otimes\H_1$. Moreover, $\pi$ is
 irreducible (e.g., by \cite[Section 4.4, Theorem 6]{vin89}). Note
 that the order of indices is in agreement with the formulation of the
 kernel theorem in Theorem~\ref{thm-kernel-2}.

If we interpret the simple tensor  $\Psi= \psi _2 \otimes \psi _1$ as
the rank one operator $f\mapsto \psi_1(f) \psi_2$  with $\psi_1\in
\H^\prime_1$, 
then we can write \eqref{prod-rep-tens} as  
$$
\pi(g)(\Psi)(f)=(\pi_1^\prime(g_1)\psi_1)(f)\cdot \pi_2(g_2)\psi_2=\big(\pi_2(g_2)\psi_2\otimes \pi_1^\prime(g_1)\psi_1\big)(f),
$$ 
where the contragredient representation $\pi^\prime_1:G_1\rightarrow
GL(\H^\prime_1)$ of $\pi_1$ is defined as
$(\pi_1^\prime(g_1)\psi_1)(f)=\psi_1(\pi_1(g^{-1}_1)f)$, see
\cite[Section 3.1]{vin89}. 

In case we treat the tensor product as a space of Hilbert-Schmidt operators, $\pi$ acts on $A\in \HS(\H_1,\H_2)$ as
$$
\pi(g)A=\pi_2(g_2)A\pi_1(g_1)^\ast.
$$

The generalized wavelet transform of a simple tensor $f_2\otimes f_1$ with respect to a ``wavelet'' $\Psi=\psi_2\otimes\psi_1$ is given by
\begin{align}\label{wavelet-factors}
V_{\Psi} (f_2\otimes f_1)(g)&=\langle f_2\otimes f_1,(\pi_2(g_2)\otimes\pi_1(g_1))(\psi_2\otimes\psi_1)\rangle \nonumber
\\ &=
\langle f_2,\pi_2(g_2)\psi_2\rangle \overline{\langle f_1,\pi_1(g_1)\psi_2\rangle}\\
&=V_{\psi_2}f_2(g_2)\overline{V_{ \psi_1}f_1(g_1)}.\nonumber
\end{align}
Thus, the wavelet transform of the tensor product representation factors into the product of wavelet transforms on $G_1$ and $G_2$. Strictly speaking, we would have to write $V_{\psi_i}^{\pi_i}f_i$ to indicate the underlying representation, but we omit the reference to the group to keep notation simple.

Throughout this paper we  consider only separable weights
$w:G\rightarrow \R_+$ with $w(g)= (w_1\otimes w_2)(g) = w_1(g_1)
w_2(g_2)$, and $m(g)= (m_1\otimes m_2)(g) = m_1(g_1) m_2(g_2)$,
where $w_i$ is submultiplicative and $m_i$ is $w_i$-moderate.  
Moreover we write $(1/w)(g)=(w_1\otimes w_2)(g)^{-1}$. 
It follows from \eqref{wavelet-factors} that the tensor representation
$\pi _2 \otimes \pi _1$ of two square-integrable representations is
again square-integrable and that the 
tensor $\Psi = \psi_2
\otimes \psi _1$ of two admissible vectors $\psi_2$ and $\psi_1$ is
admissible for $\pi$.  Likewise, if $w=w_1 \otimes w_2$ and  $\psi_1\in\mathcal{A}_{w_1}(G_1)$,
$\psi_2\in\mathcal{A}_{w_2}(G_2)$, then $\psi _2 \otimes \psi _1 \in
\mathcal{A}_{w}(G_1\times G_2)$ (where  we assume that $w_i$, $i=1,2$, satisfies \eqref{assumpt-on-w}).
Therefore all definitions and results of Section~\ref{subsec:coorbit}
hold for the representation  $\pi = \pi _2 \otimes \pi _1$ and $\Psi =
\psi _2 \otimes \psi _1$  . In particular, the orthogonality relation
\eqref{ortho-rel},  the inversion formula \eqref{inversion},
Proposition~\ref{props-of-coorbit} and Theorem~\ref{prop-equ-norms}
hold for suitable admissible vectors $\Psi=\psi_2\otimes \psi_1$.

\section{Kernel Theorems} \label{sec:kernel} 

In this section we derive the general kernel theorems for operators
between coorbit spaces. The basic idea comes from linear algebra,
where a linear operator is identified with its matrix with respect to
a basis. In coorbit theory the basic structure consists of the vectors
$\pi (g)\psi$.  Thus in analogy to linear algebra  we try to describe an operator $A: \H _1
\to \H _2$ by the kernel (= continuous matrix)
\begin{equation}
\label{essential-kern} 
  k_A(g_1,g_2) =  \langle A\pi_1(g_1)\psi_1,\pi_2(g_2)\psi_2\rangle.
\end{equation}
This can be seen as a continuous Galerkin like representation of the operator $A$ \cite{xxl08,xxlgro14}.
The idea goes back to coherent state theory~\cite[Chpt.~1.6]{perelomov86}. One of
its goals is to associate to every operator $A$ a function or symbol $k_A$, and
\eqref{essential-kern} is one of the many possibilities to do so. 

 Assume that $A:\Ci{1}\rightarrow \CiiP{\infty}$, and
$f\in \Ci{1}$, i.e., $A$ maps ``test functions'' to ``distributions''.  
By using  the inversion formula \eqref{inversion} for $f$ and applying
$A$ to it, it follows that formally 
$$
Af=\int_{G_1}\langle f,\pi_1(g_1)\psi_1\rangle A\pi_1(g_1)\psi_1 dg_1,
$$
and furthermore 
\begin{align}
V_{\psi_2} (Af)(g_2)&=\langle Af,\pi_2(g_2)\psi_2\rangle =\int_{G_1}\langle f,\pi_1(g_1)\psi_1\rangle \langle A\pi_1(g_1)\psi_1,\pi_2(g_2)\psi_2\rangle  dg_1
\notag \\
&=\int_G\langle
f,\pi_1(g_1)\psi_1\rangle k_A(g_1,g_2)dg_1 \, . \label{eq:ji1}
\end{align}
Let
\begin{equation}\label{A-tilde}
\mathfrak{ A}F(g_2)=\int_{G_1} F(g_1)k_A(g_1,g_2) dg_1    
\end{equation}
be the integral operator with the kernel $k_A$. Then \eqref{eq:ji1}
can be written as
\begin{equation}
  \label{eq:gha}
V_{\psi_2} Af = \mathfrak{A} V_{\psi_1}f \, , 
\end{equation}
or equivalently,
\begin{equation}
  \label{eq:gh}
A = V_{\psi_2}^* \mathfrak{A} V_{\psi_1} \, .  
\end{equation}

Using this factorization, the computation in \eqref{eq:ji1} can be given
a precise meaning on coorbit spaces.  Identity~\eqref{eq:gh} is the
heart of the kernel theorems. The  combination of  the
properties of the generalized wavelet transform
(Proposition~\ref{props-of-coorbit}) and boundedness properties of
integral operators yields powerful and very general kernel theorems.

We will  first show the existence of a generalized kernel
for operators mapping the space of test functions $\Ci{1}$ into the
distribution space $\CiiP{\infty}$. 
Subsequently, we will 
characterize continuous operators in certain subclasses.

\begin{theorem}\label{thm-kernel-2}
Let $G_1$ and $G_2$ be two locally compact groups, and $(\pi _j,\H
_j)$   
be integrable, unitary, irreducible
representations of $G_j$, 
such that $\mathcal{A}_{w_j}(G_j)
\neq \emptyset $, for  $j=1,2$. 

$(i)$ Every kernel $K \in \C
L^\infty_{1/w} (G_1 \times G_2)$ defines a unique linear operator $A: \Ci{1}\rightarrow
\CiiP{\infty} $ by means of
\begin{equation}
  \label{eq:c2}
\langle A\upsilon,\varphi\rangle 
=\langle K, \varphi\otimes \upsilon \rangle 
\end{equation}
for all  $\upsilon\in \Ci{1}$  and $\varphi\in \Cii{1}$. 
 The operator norm satisfies 
\begin{equation} \label{nequi}
\|A\|_{Op}
\asymp \|K\|_{\C L^{\infty}_{1/w}(G)},
\end{equation}
and 
\begin{equation}\label{Vpsi-of-ess-kernel}
k_A=V_\Psi K.
\end{equation}

$(ii)$ Kernel theorem:  Conversely, if  $A:\Ci{1}\rightarrow
\CiiP{\infty}$ is bounded, then there exists a unique
kernel $K\in \C L^\infty_{1/w} (G_1 \times G_2)$, such
that
 \eqref{eq:c2} holds. 

 \medskip
\end{theorem}

\proof $(i)$  Fix $K\in \C L^\infty_{1/w}(G)$ with $G=G_1\times G_2$, and let  $\upsilon\in
\Ci{1}$, $\varphi\in \Cii{1}$ be arbitrary. By \eqref{wavelet-factors} it follows
that $\varphi\otimes\upsilon\in \C L^1_{w}(G)$. Therefore, the duality
in \eqref{eq:c2} is well-defined and  
\begin{align}
|\langle K,\varphi\otimes \upsilon\rangle|
&\leq \|K\|_{\C L^\infty_{1/w}(G)}\|\varphi\otimes\upsilon\|_{\C L^1_{w}(G)}
\notag \\
&=\|K\|_{\C
            L^\infty_{1/w}(G)}\|\varphi\|_{\Cii{1}}\|\upsilon\|_{\Ci{1}}. \label{f3}
\end{align}
Therefore, if we fix $\upsilon$, the mapping $\varphi\mapsto \langle
K,\varphi\otimes\upsilon\rangle$ is a bounded, anti-linear
functional on $\Cii{1}$, which we call $A\upsilon\in
\CiiP{\infty}$. The map $\upsilon\mapsto A\upsilon$ is clearly linear,
and \eqref{eq:c2} defines a  linear operator
$A:\Ci{1}\rightarrow\CiiP{\infty}$. The estimate \eqref{f3} implies
that
$$
\|Av\|_{\CiiP{\infty}} \leq 
\|K\|_{\C
  L^\infty_{1/w}(G)} \|v\|_{\Ci{1}} \, ,$$ and thus
$$\|A\|_{\mathrm{Op}} \leq 
\|K\|_{\C
  L^\infty_{1/w}(G)} \, .
$$

$(ii)$ To prove the converse, we need to  show that the mapping
$K\mapsto A$ is one-to-one and onto.  

\emph{Uniqueness}: Let us assume that the kernel $\mathcal{K}\in \C
L^\infty_{1/w}(G)$  also satisfies 
$$
\langle A\upsilon,\varphi\rangle=\langle K,\varphi\otimes \upsilon\rangle=\langle \mathcal{K},\varphi\otimes \upsilon\rangle,
$$
for every $\upsilon\in\Ci{1}$, $\varphi\in\Cii{1}$.
 By Theorem~\ref{prop-equ-norms}, there exists a discrete set $\{\gamma_i\}_{i\in\mathcal{I}}\subset G$ such that
every $F\in\C L_w^1(G)$ can be written as
$$
F=\sum_{i\in\mathcal{I}}\lambda_i(F)\pi( \gamma_i)\big(\psi_2\otimes\psi_1\big),
$$ 
with unconditional  convergence in $\C L_w^1(G)$ and $\sum _i |\lambda 
_i| w(\gamma _i) \leq C \|F\|_{\C  L^1_w(G_1\times G_2)}$. Since $\pi(
\gamma_i)\big(\psi_2\otimes\psi_1\big) = \pi (\gamma _{i,2}) \psi _2
\otimes \pi ( \gamma_{i,1}) \psi _1$, we  conclude that 
\begin{align*}
\langle
  K,F\rangle&=\sum_{i\in\mathcal{I}}\overline{\lambda_i(F)}\langle
              K,\pi (\gamma _{i,2}) \psi _2
\otimes \pi ( \gamma_{i,1} \psi _1) \rangle \\ 
&=\sum_{i\in\mathcal{I}}\overline{\lambda_i(F)}\langle \mathcal{K},\pi (\gamma _{i,2}) \psi _2
\otimes \pi ( \gamma_{i,1} \psi _1) \rangle \\ 
&=\langle \mathcal{K},F\rangle.
\end{align*}
As this equality holds for every $F\in\C L^1_w(G)$, it follows that $K=\mathcal{K}$. 

\emph{Surjectivity}:
Let us assume that $A:\Ci{1}\rightarrow
\CiiP{\infty}$ is bounded, then the kernel  $k_A$  defined in
\eqref{essential-kern}   is an element of $L^\infty_{1/w}(G_1 \times
G_2)$, because  
\begin{align*}
|k_A(g)|&=\left|\langle A \pi_1(g_1)\phi,\pi_2(g_2)\psi\rangle
\right|\\
&\leq \|A\|_{Op}
\, \, \|\pi_1(g_1)\phi\|_{\Ci{1}}\,\, \|\pi_2(g_2)\psi\|_{\Cii{1}}
\\
&\leq \|A\|_{Op}
\,\,  w_1(g_1)\,\,  \|\phi\|_{\Ci{1}}\,\,  w_2(g_2)\,\,  \|\psi\|_{\Cii{1}}.
\end{align*}
We claim that $k_A$ is a generalized wavelet transform. Precisely,
there exists  $K\in \C L^\infty_{1/w}(G_1 \times G_2)$ such that
$k_A=V_{\Psi} K$. To prove this claim, we use
Proposition~\ref{props-of-coorbit}\eqref{corr-princ}, which asserts
that $k_A = V_\psi K$ for some   $K\in \C L^\infty_{1/w}(G)$  if and
only if  $ k_A=k_A\ast V_\Psi\Psi$.


As $k_A\cdot V_\Psi(\pi(g)\Psi)\in L^1(G_1\times G_2)$, we may choose
the most convenient order of integration and apply the reproducing
formula of  Proposition~\ref{props-of-coorbit} \eqref{corr-princ}
consecutively  to the representations $\pi_1$ and $\pi_2$. Using
\eqref{wavelet-factors}  we obtain
\begin{align*}
(k_A\ast V_\Psi\Psi)(g)&=\int_{G}k_A(h)  V_\Psi\Psi (h^{-1}g) \, dh \\
&=\int_{G_1}\int_{G_2} V_{\psi_2}\big(A\pi_1(h_1)\psi_1\big)(h_2) V_{\psi_2}\psi_2(h_2^{-1}g_2) dh_2\ \overline{V_{\psi_1}\psi_1(h_1^{-1}g_1)} dh_1
\\
&=\int_{G_1}\Big(V_{\psi_2}\big(A\pi_1(h_1)\psi_1\big)\ast V_{\psi_2}\psi_2\Big)(g_2) \ \overline{V_{\psi_1}\psi_1(h_1^{-1}g_1)} dh_1
\\
&=\int_{G_1}\langle A \pi_1(h_1)\psi_1,\pi_2(g_2)\psi_2\rangle\
     \overline{V_{\psi_1}\psi_1(h_1^{-1}g_1)} dh_1 = (\ast ) \, .
\end{align*}

At this point we note that by assumption on $A$  there exists a unique
operator $A^\prime:\Cii{1}\rightarrow\CiP{\infty}$ that satisfies
$$
\langle A\upsilon,\varphi\rangle =\langle \upsilon,A^\prime\ \varphi\rangle,
$$
for every $\upsilon\in \Ci{1}$ and $\varphi\in\Cii{1}$. By its definition, $A^\prime$ is weak$^\ast$-continuous.
We continue with the integration over $G_1$ and obtain 
\begin{align*}
  (\ast ) &=\overline{\int_{G_1}\langle A^\prime
            \pi_2(g_2)\psi_2,\pi_1(h_1)\psi_1\rangle\
            V_{\psi_1}\psi_1(h_1^{-1}g_1) dh_1} \\
&=\overline{\Big( V_{\psi_1}\big(A^\prime\pi_2(g_2)\psi_2\big)\ast V_{\psi_1}\psi_1\Big)(g_1)}=\overline{\langle A^\prime \pi_2(g_2)\psi_2,\pi_1(g_1)\psi_1\rangle }
\\
&=\langle A \pi_1(g_1)\psi_1,\pi_2(g_2)\psi_2\rangle=k_A(g).
\end{align*}
By Proposition~\ref{props-of-coorbit}~\eqref{corr-princ} there exists  a kernel $K\in\C
L^\infty_{1/w}(G_1 \times G_2)$, such that $k_A(g_1,g_2) 
$ $= V_\Psi K(g_1,g_2) $. 
By the first part of the proof $K$ defines an  operator $B: \Ci{1}$ $\rightarrow
\CiiP{\infty}$ by means of $\langle Bv, \phi \rangle =\langle K, \phi
\otimes v \rangle $. In particular,
\begin{align*}
\langle B
\pi_1(g_1)\psi _1 ,\pi_2(g_2)\psi _2\rangle &= \langle K, \pi_2(g_2)\psi_2 \otimes
    \pi_1(g_1) \psi _1\rangle = V_\Psi K(g_1,g_2) \\
  &= k_A (g_1,g_2) = \langle A \pi_1(g_1)\psi _1 ,\pi_2(g_2)\psi _2\rangle \, .
\end{align*}

Consequently, $B\pi_1(g_1)\psi _1 = A\pi_1(g_1)\psi _1$ for all
$g_1\in G_1$. This identity extends to all finite linear combinations
of vectors $\pi _1(g_1)\psi _1$ and  by Theorem~\ref{prop-equ-norms} to $\mathcal{C}o
_{\pi _1} L^1_{w_1}(G_1)$. Thus $B=A$ and we have shown that the map
from kernels to operators is onto. 


The map $K \mapsto A$ is bounded and invertible. By the inverse
mapping theorem we obtain that $\|K\| _{ \C
L^\infty_{1/w} (G) } \leq C \|A\|_{\mathrm{Op}}$,
  which proves \eqref{nequi}. 
\hfill $\Box$

\begin{remark} {\rm
It is crucial to interpret the brackets in \eqref{eq:c2} correctly.
For utmost precision, we would have to write
\begin{equation*}
\langle A\upsilon,\varphi\rangle_{\CiiP{\infty},\Cii{1}} 
=\langle K, \varphi\otimes \upsilon \rangle_{\C L^\infty_{1/w }(G),\C
  L^1_{w}(G)} \, ,
\end{equation*}
but we feel that this  notation would  distract from the  analogy to distribution theory.}
\end{remark}

The injectivity of the mapping $K\mapsto A$ from kernels to operators
is closely related to an important property of the coorbit spaces
$\Ci{1}$. This so-called tensor product property has gained
considerable importance in certain special cases~\cite[Theorem
7D]{fei81} and \cite{losert80}, we therefore  state and prove a general
version. Recall that the projective tensor product  of
two Banach spaces $B_1$ and $B_2$ is defined to be
$$
B_1\widehat{\otimes}B_2 = \{f=\sum_{i\in\mathcal{I}} \phi _i \otimes
\psi _i: \phi _i\in B_1, \psi _i \in B_2 \text{ and }
\sum_{i\in\mathcal{I}}\|\phi _i\|_{B_1}\|\psi _i\|_{B_2} 
<\infty  \} \, .
$$
 The norm is given as  $ \|f\|_{\widehat{\otimes}}=\inf \sum_{i\in\mathcal{I}}\|\phi
_i\|_{B_1}\|\psi _i\|_{B_2}$ over all representations of $
f=\sum_{i\in\mathcal{I}}\phi _i\otimes \psi_i$.

 The following identification of the projective tensor product of
 $\Ci{1}$ and $\Cii{1}$ with the coorbit space  $\C L^1_w(G_1\times
 G_2)$ is a generalization of  Feichtinger's original result for
 modulation spaces \cite[Theorem 7D]{fei81}. 
\begin{theorem} Under the general assumptions on the groups $G_i$ and the
representations $(\pi _i,\H_i)$ we have 
$$
\C L^1_w(G_1 \times G_2) =\Cii{1}\widehat{\otimes}\ \Ci{1}.
$$
\end{theorem}
\begin{proof}
Let $F\in\C L^1_w(G)$. Then by Theorem~\ref{prop-equ-norms} applied to $\pi = \pi _2
\otimes \pi _1$, $F$ possesses the representation
 $F=\sum_{i\in\I} \lambda_i(F)\pi(\gamma_i)\Psi\in\C L^1_w(G) $ with $\gamma
 _i = (\gamma _{i,1} , \gamma _{i,2}) \in G_1\times G_2$ and
$\sum_{i\in\I} |\lambda _i| w(\gamma _i) \leq C
\|F\|_{\C L^1_w(G)}$. Using Proposition~\ref{props-of-coorbit}~$(ii)$ we obtain that 
\begin{align*}
  \sum_{i\in\I} \| & \lambda_i(F) \pi_1(\gamma_{i,1})\psi_1
  \|_{\Ci{1}}\|\pi_2(\gamma_{i,2})\psi_2\|_{\Cii{1}}  \\
 &\leq \sum_{i\in\I} |\lambda_i(F)|w_1(\gamma _{i,1}) w_2(\gamma _{i,2}) \|\psi_1
 \|_{\Ci{1}}\|\psi_2\|_{\Cii{1}} 
 \leq C\|F\|_{\C L^1_w(G)}.
\end{align*}
Thus $F\in  \Cii{1}\widehat{\otimes}\ \Ci{1}$,  and  $\C L^1_w(G)
$ is continuously embedded into $\Cii{1}\widehat{\otimes}\ \Ci{1}$.  

Conversely, let $F \in \Cii{1}\widehat{\otimes}\ \Ci{1}$.  Choose a
representation $F = \sum \limits_{i\in\mathcal{I}} f_{i,2} \otimes
f_{i,1}$ with $\sum \limits_{i\in\mathcal{I}} \|
f_{i,1} \|_{\Ci{1}} \|f_{i,2}\|_{\Cii{1}} < \infty$. Using Fubini's theorem and
Proposition~\ref{props-of-coorbit}~$(ii)$  yields
\begin{align*} 
  \|F\|_{\C L^1_w(G)}&=
  \int_G \left| V_{\Psi} F (g) \right| w(g) dg  \\
& \leq \sum \limits_{i\in\mathcal{I}} \left( \int_{G_1} \left| V_{\psi_1} f_{i,1} (g_1) \right| w_1(g_1) dg_1 \right) \cdot \left( \int_{G_2} \left| V_{\psi_2} f_{i,2} (g_2) \right| w_2(g_2) dg_2 \right)  
\\ &= \sum \limits_{i\in\mathcal{I}} \|f_{i,1}\|_{\Ci{1}} \| f_{i,2} \|_{\Cii{1}} < \infty.
\end{align*}
Thus,  $ \Cii{1}\widehat{\otimes}\ \Ci{1}\subseteq \C L^1_w(G)$. The
equivalence of the norms follows from the inverse mapping theorem. 
\end{proof}

Once the kernel theorem provides a general  description  of operators
between test functions and distributions, we may
try to characterize certain classes of operators by properties of their
kernel. Since on  the level of the generalized wavelet transform such
operators correspond to integral operators (see diagram), we may translate the
various versions of Schur's test  to kernel theorems for operators
between coorbit spaces. Following  the procedure in \cite[Theorem
3.3]{coni17}, we first formulate a general version of Schur's test and
then derive abstract kernel theorem.

We  introduce two classes of mixed norm spaces. For two
$\sigma$-finite measure spaces $(X,\mu)$ and $(Y,\nu)$, $1\leq
p\leq\infty$, and $ m:X\times Y\rightarrow \R_+,$   we define the
spaces $L^{p,\infty}_m(X\times Y)$, and
$\mathcal{L}^{p,\infty}_m(X\times Y)$,   by the norms 
\begin{equation}\label{eq:mixed-1}
\|F\|_{L^{p,\infty}_m(X\times Y)}:=\esssup_{y\in Y}\left(\int_X |F(x,y)|^pm(x,y)^p d\mu(x)\right)^{1/p},
\end{equation}
and
\begin{equation}\label{eq:mixed-2}
\|F\|_{\mathcal{L}^{p,\infty}_m(X\times Y)}:=\esssup_{x\in X}\left(\int_Y |F(x,y)|^pm(x,y)^p d\nu(y)\right)^{1/p}.
\end{equation}

The following version of Schur's test is folklore and  can be found in
\cite[Proposition 5.2 and 5.4]{taonotes} or \cite{mad80}.




\begin{proposition}\label{tao-prop-1}
Let $(X,\mu)$ and $(Y,\nu)$ be $\sigma$-finite measure spaces,  $1\leq p\leq\infty$,
$\frac{1}{p}+\frac{1}{q}=1$, and   let $T$ be the integral operator $
Tf(y)=\int_X f(x)k_T(x,y)d\mu(x) $ with kernel $k_T:X\times
Y\rightarrow \mathbb{C}$. 

$(i)$ The operator $T$ is bounded from  $L^1_{m_1}(X)$ to 
$L^p_{m_2}(Y)$, if and only if $k_T \in \mathcal{L}^{p,\infty}_{m_1^{-1}\otimes
    m_2}(X\times Y)$. 
In that case  
\begin{equation}
\|T\|_{L^1_{m_1}(X)\rightarrow L^p_{m_2}(Y)} 
=\|k_T\|_{\mathcal{L}^{p,\infty}_{m_1^{-1}\otimes
    m_2}(X\times Y)} \, .
\end{equation}

$(ii)$ The  operator $T$ is bounded from  $L^p_{m_1}(X)$ to
$L^\infty_{m_2}(Y)$, if and only if $k_T \in L^{q,\infty}_{m_1^{-1}\otimes
    m_2}(X\times Y)$. 
In this  case 
\begin{equation}
\|T\|_{L^p_{m_1}(X)\rightarrow L^\infty_{m_2}(Y)} 
=\|k_T\|_{L^{q,\infty}_{m_1^{-1}\otimes
    m_2}(X\times Y)}. 
\end{equation}
\end{proposition}


 We now  characterize the boundedness of  operators between certain
 coorbit spaces. 
\begin{theorem}\label{thm-kernel-char1}
Let $1\leq p,q \leq \infty $ with $\frac{1}{p}+\frac{1}{q}=1$, and
$m_j$ be $w_j$-moderate weights on $G_j$. If $A$ is a bounded
operator from $\Ci{1} $ to $
\CiiP{\infty}$ with kernel $K$, then the following holds:

$(i)$  $A$ is bounded from $\Cim{1}$ to $ \Ciim{p}$, if and only if its   kernel $K$  is in $\C
\mathcal{L}^{p,\infty}_{m_1 ^{-1}\otimes m_2}(G_1\times G_2)$. 
Its  operator norm satisfies 
$$
\|A\|_{Op} \asymp \|K\|_{\C \mathcal{L}^{p,\infty}_{m_1^{-1}\otimes m_2}(G)} \, .
$$
$(ii)$  $A$ is bounded from $\Cim{p}$ to $ \Ciim{\infty}$, if and only
if  its kernel $K$ is in $\C L^{q,\infty}_{m_1^{-1}\times
  m_2}(G_1\times G_2)$. Its operator norm satisfies
$$
\|A\|_{Op}\asymp \|K\|_{\C L^{q,\infty}_{m_1^{-1}\otimes m_2}(G)}\, .
$$
\end{theorem}



\proof 
Since $\mathcal{C}o_{\pi_1}L^{1}_{w_1}(G_1)\,\subseteq
\mathcal{C}o_{\pi_1}L^{1}_{m_1}(G_1)\,$ and  $\mathcal{C}o_{\pi_2}L^{p}_{m_2}(G_2)\,\subseteq
\mathcal{C}o_{\pi_2}L^{\infty}_{1/w_2}(G_2)\,$ by \eqref{eq:f1}, the
kernel theorem is applicable to the operator $A$,  and  there exists a 
kernel $K\in  \C
L^\infty_{1/w} (G_1\otimes G_2)$, such that $V_\Psi K (g_1,g_2) = k_A(g_1,g_2) =  \langle A \pi _1(g_1)\psi
_1, \pi _2(g_2)\psi _2 \rangle$.  

 Assume first  that $K\in \C \mathcal{L}^{p,\infty}_{m_1^{-1}\otimes
  m_2}(G)$, which means that $V_\Psi K\in
\mathcal{L}^{p,\infty}_{1/m_1\otimes m_2}(G)$.
By Proposition~\ref{tao-prop-1}, the integral operator  $\mathfrak{A}$
defined by the integral kernel $k_A$ is   bounded from
$L^1_{m_1}(G_1)$ to $L^p_{m_2}(G_2)$. According to \eqref{eq:gh},  $A$
factors as $A = V_{\psi _2}^* \mathfrak{A} V_{\psi _1}$,  where  $V_{\psi
  _1}$ is an isometry from $\mathcal{C}o _{\pi _1}L^1_{m_1}(G_1)$ to
$L^1_{m_1}(G_1)$, and $V_{\psi _2}^*$ is bounded from $L^p_{m_2}(G_2)$
to $\mathcal{C}o _{\pi _2 } L^p_{m_2}(G_2)$ by
Proposition~\ref{props-of-coorbit}.  Consequently
$A$ is bounded from $\Cim{1}$ to $ \Ciim{p}$. The boundedness estimate
follows from
\begin{align*}
\|A\|_{Op} &\leq \|V_{\psi _2}^*\|_{Op} \, \| \mathfrak{A}\|_{L^1_{m_1}(G_1)\rightarrow
    L_{m_2}^p(G_2)} \, \|V_{\psi _1}\|_{Op}\\ &\leq C \|k_A\|_{\mathcal{L}^{p,\infty}_{m_1^{-1}\otimes
    m_2}(G)} = C \|K\|_{\C\mathcal{L}^{p,\infty}_{m_1^{-1}\otimes
                 m_2}(G)}. 
\end{align*}


Conversely, let $A$ be bounded from  $\Cim{1}$ to $ \Ciim{p}$. Then $A
\pi _1(g_1)\psi _1$ $ \in \mathcal{C}o _{\pi _2} L^p_{m_2}(G_2)$ and the
following estimates make sense: 
\begin{align*}
\|K\|_{\C \mathcal{L}_{m_1^{-1}\otimes m_2}^{p,\infty}(G)}&= \|V_\Psi
    K\| _{\mathcal{L}_{m_1^{-1}\otimes m_2}^{p,\infty}(G)} = \|k_A\|_{
\mathcal{L}_{m_1^{-1}\otimes m_2}^{p,\infty}(G)}\\
 &=\sup_{g_1\in G_1}\left(\int_{G_2}|\langle A\pi_1(g_1)\psi_1,\pi_2(g_2)\psi_2\rangle
 m_2(g_2)|^p dg_2\right)^{1/p}   m_1(g_1)^{-1} \\
&=\sup_{g_1\in G_1}\|A\pi_1(g_1)\psi_1\|_{\Ciim{p}}  m_1(g_1)^{-1}.\\ 
&\leq \|A\|_{Op}\sup_{g_1\in G_1}\|\pi_1(g_1)\psi_1\|_{\Cim{1}}  m_1(g_1)^{-1}
\end{align*}
Since $V_{\psi _1} \psi _1 \in L^1_{w_1}(G_1)$ and  $m_1$ is $w_1$-moderate and thus satisfies $m_1(g_1h)
m_1(g_1)^{-1} \leq w_1(h)$, the last expression is bounded
by 
\begin{align*}
  \sup_{g_1\in G_1} \|\pi_1(g_1)&\psi_1\|_{\Cim{1}}\cdot m_1(g_1)^{-1}
  =\sup_{g_1\in G_1}\int_{G_1}|\langle
    \psi_1,\pi_1(g_1^{-1}h)\psi_1 \rangle| \frac{m_1(h)}{ m_1(g_1)}dh \\ 
 &\leq \sup_{g_1\in G_1}\int_{G_1}|\langle \psi_1,\pi_1(h)\psi_1
   \rangle \frac{m_1(g_1h)}{ m_1(g_1)}dh= \|V_{\psi_1}\psi_1\|_{L^1_{w_1}(G_1)}.
\end{align*}
Thus $K\in \C \mathcal{L}_{m_1^{-1}\otimes m_2}^{p,\infty}(G)$. 

Part $(ii)$ follows by using 
Proposition~\ref{tao-prop-1}~$(ii)$ instead of $(i)$ and is proved
similarly. 
\pbox

The following diagram shows the connection between the different operators and spaces.

\begin{center}
\begin{tikzpicture}[scale=1.2]
\draw[->,thick] (-0.2,1.8)--(-0.2,-0.6)node[pos=0.5,left]{$V_{\psi_1}$};
\draw[->,thick] (1.3,-1.1)--(4.7,-1.1)node[pos=0.48,above]{$k_A\ \in\ \mathcal{L}^{p,\infty}_{m_1^{-1}\otimes m_2}(G)$};
\draw[->,thick] (1.3,2.3)--(4.7,2.3)node[pos=0.3,below]{$A$ bounded};
\draw[->,thick] (1.3,-1.1)--(4.7,-1.1)node[pos=0.3,below]{$\mathfrak{A}$ bounded};
\draw[->,thick] (6.2,1.8)--(6.2,-0.6)node[pos=0.5,right]{$V_{\psi_2}$};
\draw[->,thick] (1.66,0.25)--(1.66,-0.4)node[pos=0.5,right]{$V_{\Psi}$};
\draw[<->,thick] (1.66,1.7)--(1.66,0.9);
\draw (0,-1.1) node {$L^1_{m_1}(G_1)$};
\draw (0,2.3) node {$\Cim{1}$};
\draw (6,2.3) node {$\Ciim{p}$};
\draw (6,-1.1) node {$L^p_{m_2}(G_2)$};
\draw (3.21,0.5) node {$K\ \in\ \C \mathcal{L}^{p,\infty}_{m_1^{-1}\otimes m_2}(G)$};
\end{tikzpicture}
\end{center}


Using interpolation between $L^p$-spaces, Schur's test can also be
formulated as saying that an integral operator is bounded on all $L^p$
simultaneously, if and only if its kernel belongs to $L^{1,\infty }
\cap \mathcal{L}^{1,\infty}$. The corresponding version for coorbit
spaces is a consequence of  Theorem~\ref{thm-kernel-char1} and  an interpolation argument.
\begin{corollary}
The following conditions are equivalent:
\begin{enumerate}[(i)]
\item $A:\Cim{p}\rightarrow \Ciim{p}$ is bounded for every $1\leq p\leq\infty$.
\item Both  $A:\Cim{1}\rightarrow \Ciim{1}$ and $A:\Cim{\infty}\rightarrow \Ciim{\infty}$ are bounded.
\item $K \in \C \mathcal{L}^{1,\infty}_{ m_1^{-1}\otimes m_2}(G)
  \bigcap K \in \C L^{1,\infty}_{m_1^{-1}\otimes m_2}(G)$.
\end{enumerate}
\end{corollary}

Clearly one can now translate every boundedness result for an integral
operator into a kernel theorem for coorbit spaces. As a simple, but
important example we offer a 
sufficient condition for regularizing
operators, i.e., of operators that map distributions to test
functions.

\begin{theorem} \label{regularize}
  Under the assumptions of Theorem~\ref{thm-kernel-2}, if the unique kernel of the operator $A$ satisfies $K\in
    \mathcal{C}o _\pi L^1_{ w} (G)$, then $A$ is bounded 
  from $ \mathcal{C}o _{\pi _1} L^\infty _{1/w_1} (G_1)$ to  $\mathcal{C}o
  _{\pi _2} L^1_{w_2} (G_2)$.
  \end{theorem}
  \begin{proof}
    Consider the integral operator $\mathfrak{A}$ as in the proof of Theorem~\ref{thm-kernel-char1} and observe $V_\Psi K=k_A\in L^1_{w}(G)$ is a sufficient condition for $\mathfrak{A}:L^\infty_{1/w_1}(G_1)\rightarrow L^1_{w_2}(G_2)$ to be bounded by Schur's test.
  \end{proof}

\subsection{Discretization}
Coorbit theory guarantees the discretization of the coorbit spaces via atomic decompositions and Banach frames.
For our purposes, it is sufficient to state a shortened and simplified
version of \cite[Theorem 5.3]{groe91}. Let $Y$ be one of the function
spaces  $L^p_m(G),L^{p,\infty}_m(G),$
or $ \mathcal{L}^{p,\infty}_m(G)$, and $Y_d$ the natural sequence
space associated to $Y$. 
\begin{proposition}\label{Xd-def}
If $\psi$ satisfies 
\begin{equation}\label{Bw}
\int_G \sup_{h\in gQ}|V_\psi\psi(h)|w(g)dg<\infty,
\end{equation}
for a compact neighborhood $Q$ of $e$, then there exists a discrete subset $\Lambda\subset G$, and constants $C_1,C_2>0$, such that 
\begin{equation}\label{xd-frame}
C_1\|f\|_{\C Y}\leq\|V_\psi f\|_{Y_d}\leq C_2\|f\|_{\C Y}, \qquad
\text{ for every } f\in\C Y \, . 
\end{equation}
\end{proposition}

\begin{corollary}\label{discrete-cor}
Let $\Lambda=\Lambda_1\times\Lambda_2\subset G$ a discrete set such
that $\{\pi(\lambda)\Psi\}_{\lambda\in\Lambda}$ satisfies
\eqref{xd-frame} for  $\C L_{m_1^{-1}\otimes m_2}^{p,\infty}(G)$ and
$\C \mathcal{L}_{m_1^{-1}\otimes m_2}^{p,\infty}(G)$. If $A$ is a bounded
operator from $\Ci{1} $ to $
\CiiP{\infty}$ with kernel $K$, then  the following holds: \\
$(i)$ $A:\mathcal{C}o _{\pi _1} L^1_{m_1}(G_1) \rightarrow \mathcal{C}o _{\pi _2}
L^p_{m_2}(G_2)$ is bounded  if and only
if  \begin{equation}\label{K_in_ell^p,infty} 
\sup_{\lambda_1\in \Lambda_1}\left(\sum_{\lambda_2\in\Lambda_2}|V_{\Psi}K(\lambda)(m_1^{-1}\otimes m_2)(\lambda)|^p\right)^{1/p}<\infty.
\end{equation}
$(ii)$ Likewise  $A:\mathcal{C}o _{\pi _1} L^p_{m_1}(G_1) \rightarrow \mathcal{C}o _{\pi _2}
L^\infty _{m_2}(G_2)$ is bounded if and only if  \begin{equation}\label{K_in_ell^q,infty}
\sup_{\lambda_2\in \Lambda_2}\left(\sum_{\lambda_1\in\Lambda_1}|V_{\Psi}K(\lambda)(m_1^{-1}\otimes m_2)(\lambda)|^q\right)^{1/q}<\infty.
\end{equation}
\end{corollary}
\proof $(i)$ By Theorem~\ref{thm-kernel-char1} $A$ has a kernel in
$\C\mathcal{L}^{p,\infty}(G)$, and   
$\|A\|_{Op}\asymp \| K\|_{\C\mathcal{L}^{p,\infty}(G)}$. By
\eqref{xd-frame}, the expression in \eqref{K_in_ell^p,infty} is an
equivalent norm for  $\| K\|_{\C\mathcal{L}^{p,\infty}(G)}$. 
The proof of $(ii)$ works exactly the same. \pbox

\section{Examples}\label{sec:ex} 

\subsection{Modulation spaces}\label{subsec:modulation}

The  Weyl-Heisenberg group $G_{WH}=\R^d\times \R^d\times \mathbb{T}$
is defined  by the group law
 $$
 (x,\omega,e^{2\pi i\tau})\cdot(x^\prime,\omega^\prime,e^{2\pi i\tau^\prime})=(x+x^\prime,\omega+\omega^\prime,e^{2\pi i(\tau+\tau^\prime-x\cdot\omega^\prime)}).
 $$
 Let $T_xf(t):=f(t-x)$ denote the translation, and $M_\omega
 f(t):=e^{2\pi i\omega t}f(t)$ the modulation operator. The operator 
 $\pi_{WH}(x,\omega,\tau)=e^{2\pi i \tau}M_\omega T_x$ for $(x,\omega
 ,\tau) \in G_{WH}$  defines a unitary square-integrable
 representation of $G_{WH}$ acting  on $L^2(\R^d)$,  for which every
 nonzero vector in $L^2(\R^d)$ is admissible.
 Since the phase factor $e^{2\pi i \tau }$ is irrelevant for the
 definition of coorbit spaces, it is  
convenient to drop the trivial third component, and consider
the time-frequency shift $\pi(x,\omega)=\pi_{WH}(x,\omega,1)= M_\omega T_x$. Formally, we treat
the  projective representation $\pi $ of $\mathbb{R}^{2d}$ instead of
the  unitary representation $\pi _{WH}$ of $G_{WH}$. 
The  transform corresponding  to $\pi$  is nothing else but  the 
short-time Fourier transform 
$$
V_\psi f(x,\omega)=\langle f,M_\omega T_x\psi\rangle =\int_{\R^d}f(t)\overline{\psi(t-x)}e^{-2\pi i\omega t}dt,\ f,\psi\in L^2(\R^d).
$$
The coorbit spaces
associated to $\pi_{WH}$  coincide therefore with the   coorbit spaces
associated to $\pi $. These are the modulation spaces $M_m^p(\R^d)$
which   were first introduced by Feichtinger in \cite{fei83} as certain decomposition spaces and
subsequently were  identified with the coorbit
spaces of the Heisenberg group $\mathcal{C}o _{\pi _{WH}} L^p_m
(G_{WH}) = \mathcal{C}o _\pi 
L^p_m(\rdd)=M^p_m(\R^d)$~\cite{FGchui}. 
We refer to the standard textbooks \cite{fo89,groe1} for more information on time-frequency analysis.

Theorem~\ref{thm-kernel-2} asserts that every bounded operator from
$M^1_w(\R^d) \hspace{-0.01cm} =\hspace{-0.01cm}  \mathcal{C}o _{\pi _{WH}} L^1_w(G_{WH})\hspace{-0.03cm} = \C L^1_w(\R ^{2d}) $ to $M^\infty_{1/w}(\R^d) =
\mathcal{C}o _{\pi _{WH}} L^\infty_{1/w}(G_{WH}) =
\C L^\infty_{1/w}(\R ^{2d})$  possesses a
kernel $K\in  \mathcal{C}o _{\pi _{WH} \otimes\pi _{WH} } L^\infty
_{w^{-1} \otimes w^{-1} } (G_{WH} \times G_{WH})$, such that $\langle
Af,g \rangle = \langle K, g \otimes f \rangle $ for $f,g\in M^1_w(\R^d)$. 
Let us elaborate in detail what the kernel theorem asserts in this
case: for $g_i = (x_i,\omega _i, \tau _i) \in G_{WH}, i=1,2$, the
tensor representation $\pi _{WH} \otimes \pi _{WH} $ acts on the
simple tensor $(\psi _2 \otimes \psi _1)(t_2,t_1) = \psi
_2(t_2)\overline{\psi _1(t_1)} \in
L^2(\R ^d) \otimes L^2(\R ^d) \cong L^2(\R ^{2d})$ as
\begin{align*}
  \pi _{WH}\otimes  \pi _{WH} (g_2,g_1) (\psi _2\otimes \psi _1)(t_2,t_1)
  &= e^{2\pi i (\tau _1 - \tau _2)}  M_{\omega _2}T_{x_2}\psi _2(t_2)
    \overline{M_{\omega _1}T_{x_1}\psi _1(t_1)} \\
  &= e^{2\pi i (\tau _1 - \tau _2)} M_{(\omega _2,-\omega_1)}
    T_{(x_2,x_1)}(\psi _2 \otimes \psi _1)(t_2,t_1) \, .
\end{align*}
Thus except for the phase factor $ e^{2\pi i (\tau _1 - \tau _2)}$ 
the tensor representation $  \pi _{WH}\otimes  \pi _{WH}$ is just the
time-frequency shift $M_{(\omega _2,-\omega_1)}
    T_{(x_2,x_1)}$ acting on $L^2(\R ^{2d})$. Consequently, the coorbit
    spaces with respect to  the product group $G_{WH} \otimes G_{WH}$
    are again modulation spaces, this time on $\R ^{2d}$. 
    For the coorbit of $L^\infty $ we compare the norms  
$$
\|K\|_{M^\infty(\R^{2d})}=\sup_{(x_1,x_2,\omega_1,\omega_2)\in\R^{4d}}\big|\big\langle
K,M_{(\omega _1, \omega_2)} T_{(x_1,x_2)}(\psi_2\otimes\psi_1)\big\rangle\big|,
$$
and 
$$
\|K\|_{\mathcal{C}o_{\pi\otimes\pi}
  L^\infty(\R^{4d})}=\sup_{(x_1,\omega_1)\otimes(x_2,\omega_2)\in\R^{4d}}\left|\big\langle
  K,\big(\pi(x_1,\omega_1)\otimes\pi(x_2,\omega_2)\big)(\psi_2\otimes\psi_1)\big\rangle\right|, 
$$
which are obviously equal. 
 In this case   Theorem~\ref{thm-kernel-2} is therefore
    just Feichtinger's kernel theorem: \emph{For 
      $A: M^1(\R ^d) \to M^\infty (\R ^d)$  there  exists   a unique
      kernel  $K\in M^\infty(\R^{2d})$, such that $\langle Af,g\rangle =
      \langle K, g\otimes f \rangle$.}


The recent extension of Feichtinger's kernel theorem by    Cordero and
Nicola \cite{coni17} can be seen in the same light. Let us explain the
difference in the formulations. 
Our approach considers the generalized wavelet transform  $$
V_{\Psi}K(x_1,\omega_1,x_2,\omega_2)=\langle K,\pi
_{WH}(x_2,\omega_2,1)\otimes \pi _{WH}(x_1,\omega_1,1)(\psi _2 \otimes  \psi_1)\rangle,
$$
of the kernel. The conditions of Theorem~\ref{thm-kernel-char1} are
formulated by mixed norms acting simultaneously on the variables
$(x_2,\omega _2)$ and on  $(x_1,\omega _1)$. The treatment in~\cite{coni17}
uses the short-time Fourier transform on $\R ^{2d}$
$$
V_{\Psi}K(x_1,x_2,\omega_1,\omega_2)=\langle K,M_{(\omega_1,\omega_2)}
T_{(x_1,x_2)} \Psi\rangle \, ,
$$
which is  the same transform, except for the order of the
variables. In~\cite{coni17} it was therefore necessary to 
reshuffle the order of integration of time-frequency shifts and to
use the notion of mixed modulation spaces, which were  studied
in~\cite{bi10,MOP16}. The new insight of our formulation is that the
mixed modulation spaces are simply the coorbit spaces with respect to
the tensor product representation.


The special case of
Theorem~\ref{thm-kernel-char1} for the Weyl-Heisenberg group and the
weights $m_s(x,\omega,\tau )=(1+|x|+|\omega|)^s$ for $s\in \R $ states the
following:  
Fix $\sigma >0$ and  let $A$ be an operator from $M^1_{m_\sigma}(\R ^d) $ to
$M^\infty _{m_{-\sigma }}(\R ^d)$. Then for $|r|, |s| \leq \sigma $,
$1\leq p,q\leq \infty $ and $1/p+1/q=1$  we
have 

\medskip

\begin{tabular}{llll}
$(i)$ \ $A:M_{m_s}^{1}(\rd)\rightarrow M_{m_r}^{p}(\rd)$ &  bounded
  &$\Leftrightarrow$& $K\in \C \mathcal{L}^{p,\infty}_{m_{-s}\otimes
                      w_{r}}(\R^{4d})$,\\ 
$(ii)$ $A:M_{m_s}^{p}(\rd)\rightarrow M^{\infty}_{m_r}(\rd)$ & bounded &$\Leftrightarrow$& $K\in \C L^{q,\infty}_{ m_{-s}\otimes m_{r}}(\R^{4d})$.
\end{tabular}

Regularizing operators from $M^\infty$ to $M^1$ were studied recently
studied by Feichtinger and Jakobsen~\cite{feija18}:
they characterized a subclass of this space of operators by an  integral kernel in
$M^1(\R^{2d})$. The sufficiency of this result in a coorbit version is
contained in Theorem~\ref{regularize}. 

\subsection{Wavelet Coorbit Spaces and Besov Spaces}\label{subsec:wavelet}


The affine group $G_{\mathrm{aff}}=\R\times\R^*$ is given by the group law
$(x,a)\cdot(y,b)=(x+ay,ab)$ where $x,y\in \R$ and $a,b \in \R
\setminus \{0\}$.  Its left 
Haar measure is given by $\frac{dxda}{a^{2}}$. Let
$D_af(t)=a^{-1/2}f(t/a)$ denote the dilation operator. Then   $ (x,a)
\to \paff (x,a)= T_x D_a$ defines an 
unitary,
square-integrable representation 
of $G_{\mathrm{aff}}$ on $L^2(\R )$.  

Let now $f,\psi\in L^2 (\R )$. The continuous wavelet transform is defined  as
$$
W_\psi f(x,a):=\langle f,\paff(x,a)\psi\rangle=a^{-1/2}\int_\R f(t)\overline{\psi(a^{-1}(t-x))}dt,
$$
and the admissibility  condition \eqref{square-int} reads as 
$$
\int_{\R ^*}|\widehat{\psi}(\omega)|^2\frac{d\omega}{|\omega |}<\infty.
$$

 It is well-known that the coorbit
spaces associated to the representation $\paff $ are the homogeneous
Besov spaces. See the  textbooks \cite{dau92,mey93} for details and  further
expositions of  wavelet theory. For brevity, we  consider only the coorbit spaces with
respect to the  weighted $L^p(G_{\mathrm{aff}})$-spaces with the
weight function  $\nu _s(x,a)=\nu _s(a)=|a|^{-s}$ for $s\in \R $. Note that
$\nu _{-s}=1/\nu _s$. Then $\mathcal{C}o _{\paff}
L^p_{\nu _s}(G_{\mathrm{aff}})=\dot{B}_{p,p}^{s-1/2+1/p}(\R)$ by 
\cite[Section 7.2]{fegr88}. In particular $\mathcal{C}o _{\paff} L^1_{\nu _s}(G_{a})=\dot{B}_{1,1}^{s+1/2}(\R)$ and $\mathcal{C}o_{\paff}
L^\infty_{\nu _s}(G_{\mathrm{aff}})=\dot{B}_{\infty,\infty}^{s-1/2}(\R)$.
In this example  Theorem~\ref{thm-kernel-2} states that \emph{an operator 
  $A:\dot{B}_{1,1}^s(\R)\rightarrow \dot{B}_{\infty,\infty}^{-r}(\R)$ is bounded
  if and only if its associated kernel $K$ is in $ \mathcal{C}o _{\paff
    \otimes \paff } \, 
  L^{\infty}_{\nu _{-s-1/2}\otimes \nu _{-r-1/2}}(G_{\mathrm{aff}}^2)$.}
At first glance not much seems to have been gained by this 
formulation, but it turns out that the coorbit spaces of the tensor
product $\paff \otimes \paff $ of $G_{\mathrm{aff}}^2$ are well understood
in the theory of function spaces under the name of \emph{Besov spaces
  of dominating mixed  smoothness. } In particular,   $ \mathcal{C}o _{\paff
    \otimes \paff } \, 
  L^{\infty}_{\nu _{-s-1/2}\otimes \nu _{-r-1/2}}(G_{\mathrm{aff}}^2)$  can be identified with the Besov space of dominating
  mixed smoothness $S^{-s,-r}_{\infty,\infty}B(\R^2)$. See 
\cite{SU09}, Def.\ A.4 and~\cite{ST87}.
 Moreover, Theorem~\ref{thm-kernel-char1}   yields a characterization of continuous
operators between certain Besov spaces: 

\noindent \begin{tabular}{llll}
\\ $(i)$ \ $A:\dot{B}_{1,1}^{s}(\R)\rightarrow \dot{B}_{p,p}^{r}(\R)$\hspace{-0.1cm} & bounded &\hspace{-0.1cm}$\Leftrightarrow$&\hspace{-0.1cm}$K\in \mathcal{C}o_{\paff\otimes\paff}\mathcal{L}^{p,\infty}_{\nu _{-s+1/2}\otimes \nu _{r+1/2-1/p}}(G_{\mathrm{aff}}^2)$,
\\ $(ii)$ $A:\dot{B}_{p,p}^{s}(\R)\rightarrow
  \dot{B}_{\infty,\infty}^{r}(\R)$\hspace{-0.1cm} &  bounded &\hspace{-0.1cm}$\Leftrightarrow$&\hspace{-0.1cm}$K\in \mathcal{C}o_{\paff\otimes\paff} L^{q,\infty}_{\nu _{-s-1/2+1/p}\otimes \nu _{r+1/2}}(G_{\mathrm{aff}}^2)$.\\
\end{tabular}
\medskip

The case $(i)$ for $p=1$ was already formulated in a discrete version
by Meyer \cite[Section 6.9, Proposition 6]{mey93}.  
\begin{theorem}
Let
$\{\psi_{k,j}\}_{(k,j)\in\Z^2}$ be a wavelet basis with
$\psi_{k,j}(t)=2^{j/2}\psi(2^{j}t-k)$, and assume that $\psi $ has
compact support and satisfies sufficiently many moment conditions so
that the assumption of Proposition~\ref{Xd-def} is satisfied. 
An operator
$A:\dot{B}_{1,1}^0(\R)\rightarrow \dot{B}^0_{1,1}(\R)$ is bounded   if and only if 
$$
\sup_{(k\prime,j\prime)\in\Z^2}\sum_{(k,j)\in\Z^2}\big|\langle A\psi_{k\prime ,j\prime},\psi_{k,j}\rangle \big|2^{-j/2+j^\prime/2}\leq C.
$$
\end{theorem}
\proof Set $p=1$, $s=-1/2$, recall that $k_A=V_\Psi K$ and apply Corollary~\ref{discrete-cor}.\pbox

  \subsection{The Case of Two Distinct Representations}\label{subsec:distinct-rep}
  
For most applications it suffices to consider a single group $G$ and
its product group $G\times G$. Our formulation with two different
group  allows us to study operators acting between coorbit spaces
associated with different group representations. Using the
representations of the Weyl-Heisenberg group and the affine group of  Sections~\ref{subsec:modulation} and
\ref{subsec:wavelet}, one can characterize the boundedness of
operators between certain modulation spaces and Besov spaces by
properties of their associated kernels.  
Theorem~\ref{thm-kernel-char1} now reads as follows:    

\medskip
\noindent \begin{tabular}{llll}
$(i)$\ \ \ $A:M^1_{m_s}(\R^d)\rightarrow \dot{B}^r_{p,p}(\R)$ &\hspace{-0.2cm}bdd.
   &\hspace{-0.2cm}$\Leftrightarrow$&\hspace{-0.2cm}$K \in
                       \mathcal{C}o_{\paff\otimes\pi_{WH}}\mathcal{L}^{p,\infty}_{
                       \widetilde{m}_{-s}\otimes \nu
                       _{r+1/2-1/p}}(G_{WH}\times G_{\mathrm{aff}})$,\\ 
$(ii)$ \ $A:M^p_{m_s}(\R^d)\rightarrow \dot{B}^r_{\infty,\infty}(\R)$ &\hspace{-0.2cm}bdd.
   &\hspace{-0.2cm}$ \Leftrightarrow$& \hspace{-0.2cm}$K\in \mathcal{C}o_{\paff \otimes\pi_{WH}}L^{q,\infty}_{m_{-s}\otimes \nu _{r+1/2}}(G_{WH}\times G_{\mathrm{aff}})$,
\\  $(iii)$ $A:\dot{B}^r_{1,1}(\R)\rightarrow M_{m_s}^p(\R^d)$ &\hspace{-0.2cm}bdd.
   &\hspace{-0.2cm}$\Leftrightarrow$ &\hspace{-0.2cm}$K\in \mathcal{C}o_{\pi_{WH}\otimes\paff }\mathcal{L}^{p,\infty}_{m_{-r+1/2}\otimes \nu _s}(G_{a}\times G_{WH})$,
  \\ $(iv)$  $A:\dot{B}^r_{p,p}(\R)\rightarrow M_{m_s}^\infty(\R^d)$ &\hspace{-0.2cm}bdd.
   &\hspace{-0.2cm}$\Leftrightarrow$&\hspace{-0.2cm}$K\in \mathcal{C}o_{\pwh\otimes \paff }L^{q,\infty}_{\nu _{-r-1/2+1/p}\otimes m_s}(G_{a}\times G_{WH})$.\\
  \end{tabular}

  \vspace{3mm}
  
  As a special case one obtains a characterization of the
bounded operators $ A$ from $ \dot{B}^r_{1,1}(\R)$ to $ L^2(\R^d)$. Since $M^2(\R^d)=L^2(\R^d)$,   they are
completely characterized  by the
membership of their kernel in   
$\mathcal{C}o_{\paff \otimes\pi_{WH}}\mathcal{L}^{2,\infty}_{1\otimes
  m_{-r+1/2}}( G_{\mathrm{aff}}\times G_{WH})$. 
\end{document}